\documentclass[12pt, a4paper]{article}
\usepackage[latin1]{inputenc}
\usepackage[english]{babel}
\usepackage{amsmath,verbatim,amsthm,mathrsfs,stmaryrd}
\usepackage[english]{babel}
\usepackage{hyperref}
\usepackage{indentfirst}
\usepackage{amstext}
\usepackage{enumerate}
\usepackage[all]{xy}
\usepackage{amsfonts} 
\usepackage{textcomp}
\usepackage{amssymb}
\usepackage[dvips]{graphicx}
\usepackage{setspace}
\usepackage{color}
\usepackage{ulem}
\usepackage[dvipsnames]{xcolor}
\usepackage{graphicx}
\graphicspath{ {./images/} }
\usepackage[all]{xy}
\usepackage[noabbrev, capitalise, nameinlink, nosort]{cleveref}
\usepackage[mathscr]{euscript}

\newcommand{\field}[1]{\mathbb{#1}}

\newcommand {\R}{\mathbb{R}}
\newcommand {\N}{\mathbb{N}}

\newcommand{\W}{\mathfrak{W}}

\newcommand {\Z}{\mathbb{Z}}

\newcommand{\NN}{\field{N}}

\newcommand{\bgln}{\begin{eqnarray}} 
\newcommand{\egln}{\end{eqnarray}}
\newcommand{\bgl}{\begin{equation}} 
\newcommand{\egl}{\end{equation}}
\newcommand{\alf}{\mathscr{A}}
\newcommand{\osf}{{\normalfont \textsf{X}}}

\newcommand{\CZ}{\mathcal{Z}}

\usepackage{geometry}
\newcommand{\Dom}{\operatorname{Dom}}
\newcommand{\Img}{\operatorname{Im}}

\newtheorem{teorema}{theorem}[section]
\newtheorem{lemma}[teorema]{Lemma}
\newtheorem{corollary}[teorema]{Corollary}
\newtheorem{definition}[teorema]{Definition}
\newtheorem{proposition}[teorema]{Proposition}

\newtheorem{example}[teorema]{Example}
\theoremstyle{remark}
\theoremstyle{definition}
\newtheorem{remark}[teorema]{Remark}
\newtheorem{theorem}[teorema]{Theorem}

\title{Nonwandering sets and the entropy of local homeomorphisms}

\author{Daniel Gon\c{c}alves\footnote{Partially supported by Conselho Nacional de Desenvolvimento Cient\'ifico e Tecnol\'ogico (CNPq) and Funda\c{c}\~ao de Amparo \`a Pesquisa e Inova\c{c}\~ao do Estado de Santa Catarina (Fapesc) - Brazil.}, Danilo Royer, and Felipe Augusto Tasca\footnote{Partially supported by Funda\c{c}\~ao de Amparo \`a Pesquisa e Inova\c{c}\~ao do Estado de Santa Catarina (Fapesc) - Brazil.}}
\date{September 2024}

\begin{document}
\maketitle

\begin{abstract}

A local homeomorphism between open subsets of a locally compact Hausdorff space induces dynamical systems with a wide range of applications, including in C*-algebras. In this paper, we introduce the concepts of nonwandering and wandering sets for such systems and show that, under mild conditions, the metric entropy defined in \cite{DDF} is concentrated in the nonwandering set. More generally, we demonstrate that the entropy of the system is the maximum of the entropies of the systems restricted to the nonwandering set and the closure of the wandering set. We illustrate these results with several examples, including applications to subshifts over countable alphabets.

\vspace{5mm}

\end{abstract}

\section{Introduction}

Systems consisting of a locally compact Hausdorff space $X$, along with a local homeomorphism $\sigma$ between open subsets of $X$, have been extensively studied. These systems capture the intrinsic dynamics found in various contexts, including one-sided shifts of finite type over finite alphabets, edge shifts of infinite graphs and ultragraphs \cite{MR3600124,GDD2,MR3938320,BrixCarlsen1}, the full Ott-Tomforde-Willis shift over an infinite alphabet \cite{OTW,BGGV,BGGV3}, and self-covering maps \cite{armstrong2023conjugacy}, among others (see \cite{MR1814145, BACS} for more examples). When the local homeomorphism is surjective, such systems can be characterized in terms of separated graphs, see \cite{clara}.

The study of local homeomorphism systems, as described above, is deeply connected to C*-algebra theory. Deaconu and Renault demonstrated in \cite{Deaconu95} how to associate a C*-algebra with such systems, leading to the term "Deaconu-Renault system" to describe a system consisting of a locally compact Hausdorff space $X$ together with a local homeomorphism $\sigma$ between open subsets of $X$. These associated C*-algebras have been shown to capture various properties of the system. In particular, conjugacy of Deaconu-Renault systems is characterized in terms of the corresponding C*-algebras, see \cite{BACS}.

Entropy is one of the most well-known invariants for isomorphism of dynamical systems and has been extensively studied across various contexts. In the setting of Deaconu-Renault systems, the notion of entropy, both metric and topological, was introduced by the authors in \cite{DDF}, where it was shown to also serve as an invariant for conjugacy in these systems. However, because the map in Deaconu-Renault systems is only locally a homeomorphism, not all classical results about entropy have direct analogs, and those that do must be carefully extended and analyzed. In this context, the goal of this paper is to extend the concepts of wandering and nonwandering points to Deaconu-Renault systems and to explore the relationship between these sets and entropy, aiming for a generalization of the classical result that states that the entropy of a system is concentrated in the nonwandering set, see \cite{bowen1970topological, entropypartialaction}.

In the study of a homeomorphism on a compact set $(\mathcal{X}, \phi)$, the wandering set is defined as the set of points for which there exists an open neighborhood $U$ such that $\phi^n(U) \cap U = \emptyset$ for every $n \in \mathbb{N}$. This condition models the idea that a point has a neighborhood whose future never returns to it. A point is said to be nonwandering if it is not wandering. In this case, both the nonwandering and wandering sets are invariant under the homeomorphism $\phi$.

In the local homeomorphism setting, we must be more cautious. For example, consider the one-sided shift operator defined on the discrete shift space with two points $X = \{01^\infty, 1^\infty\}$. Here, $01^\infty$ is wandering and $1^\infty$ is nonwandering, so the wandering set is not invariant under the shift map. 
To address this issue, we propose a refined definition of a wandering point. We require the existence of a neighborhood whose future never intersects any possible past of a future of it. More precisely, we say that $x$ is wandering if there exists an open neighborhood $U$ of $x$ such that $\sigma^n(U) \cap \sigma^{-k}(\sigma^k(U)) = \emptyset$ for all $n \geq 1$ and $k \in \mathbb{Z}$ (see Definition~\ref{wandering}). With this definition, the nonwandering and wandering sets are invariant under the local homeomorphism (Proposition~\ref{invariantshallbe}), a property that allows us to consider the restriction of the local homeomorphism to the (non)wandering set and compute the entropy of the restricted system.

We organize our work as follows: Section 2 is divided into two parts. In the first part, we review preliminary concepts related to Deaconu-Renault systems and introduce conditions under which a restricted system is again a Deaconu-Renault system. In the second part, we explain how to compute the metric entropy of these systems, as described in \cite{DDF}. 

Section 3 focuses on the set of nonwandering points. We formally define wandering and nonwandering points in singly generated dynamical systems, and present several examples, including the nonwandering set associated with the local homeomorphism arising from the construction of the Sierpiski triangle as an iterated function system. We also describe the nonwandering set for subshifts over countable alphabets.

Finally, in Section 4, we show that, in general, the entropy of a Deaconu-Renault system is the maximum of the entropies of the systems restricted to the nonwandering set and the closure of the wandering set. Additionally, for systems defined on a compact space with a clopen domain, we prove that the entropy is concentrated in the nonwandering set. We conclude by applying these results to compute the entropy of several examples.

\section{Preliminaries}\label{sec2}

In this section, we recall and develop necessary concepts regarding Deaconu-Renault systems and their metric entropy. Throughout our work, we denote the set of natural numbers $\{1,2,\ldots\}$ by $\NN$.

\subsection{Restricted Deaconu-Renault systems}

Deaconu-Renault systems, also known as singly generated dynamical systems, were originally defined in \cite{Renault00}. In this section, we describe when a restriction of a Deaconu-Renault system is again a Deaconu-Renault system. We begin with the definition of a Deaconu-Renault system, as given in \cite{BACS}.

\begin{definition} A Deaconu-Renault system is a pair $(X,\sigma)$ consisting of a locally compact Hausdorff space $X$, and a local homeomorphism $\sigma: \Dom(\sigma)\longrightarrow \operatorname{Im}(\sigma)$ from an open set $\Dom (\sigma)\subseteq X$ to an open set $\operatorname{Im}(\sigma)\subseteq X$. 
\end{definition}

\begin{remark}\label{intersection notation}
    Let $(X,\sigma)$ be a Deaconu-Renault system and $Y$ a subset of $X$ which is not necessarily contained in $\Dom(\sigma)$. To shorten notation, we write $\sigma(Y)$ instead of $\sigma(Y\cap \Dom(\sigma))$. Similarly, for each $n\geq 1$, we write $\sigma^n(Y)$ meaning $\sigma^n(Y\cap \Dom(\sigma^n))$. For $k<0$, the notation $\sigma^k(Y)$ is the inverse image $\sigma^{k}(Y)=\{x\in \Dom(\sigma^{-k}): \sigma^{-k}(x)\in Y \}$. Additionally, instead of writing {\it Deaconu-Renault system}, we will now simply refer to it as a {\it DR system}.
    \end{remark}

To obtain restrictions of DR systems that are again DR systems, we need the notion of invariant subsets, as defined below. 

\begin{definition}\label{def:invariant}
Let $(X,\sigma)$ be a DR system. We say that $Y \subseteq X$ is $\sigma$-invariant if $\sigma(Y)\subseteq Y$ and we say that $Y$ is $\sigma^{-1}$-invariant if $\sigma^{-1}(Y)\subseteq Y$. If $Y$ is invariant under both $\sigma$ and $\sigma^{-1}$, then we simply say that $Y$ is $(\sigma, \sigma^{-1})$-invariant. 
\end{definition} 

\begin{remark}\label{remark_inv}
    In \cite[Definition 5.1]{DDF}, to obtain a DR system, the authors require the restriction set $Y$ to be open. We will see below that this hypothesis is not necessary. Moreover, notice the $Y$ being $\sigma$-invariant under Definition~\ref{def:invariant} is equivalent to $\sigma^n(Y)\subseteq Y$ for all $n\geq 1$.  
\end{remark}

We will use the following lemma to prove our result regarding the restriction of DR systems $(X,\sigma)$ to $(\sigma, \sigma^{-1})$-invariant subsets.

\begin{lemma}\label{lema_rest}
    Let $(X, \sigma)$ be a DR system.
    \begin{enumerate}
     \item Let $Y,Z\subseteq X$ be disjoint subsets such that $X=Y\sqcup Z$. Then, $Y$ is $\sigma$-invariant if and only if $Z$ is $\sigma^{-1}$-invariant.
     \item If $Y\subseteq X$ is a $(\sigma, \sigma^{-1})$-invariant set, then $\sigma_{|_Y}(Y\cap \Dom(\sigma))$ is an open subset of $Y$ (with the induced topology).
    \end{enumerate}
\end{lemma}
\begin{proof} 
 
First, suppose that \(\sigma(Y) \subseteq Y\). Fix an element \(x \in \sigma^{-1}(Z)\), so that \(\sigma(x) \in Z\). In particular, we have \(x \in \operatorname{Dom}(\sigma)\). If we suppose that \(x \in Y\) then, from our hypothesis, we would have \(\sigma(x) \in Y\), which is impossible since \(\sigma(x) \in Z\) and \(Y\) and \(Z\) are disjoint sets. Therefore, \(x \in Z\). Consequently, we get that \(\sigma^{-1}(Z) \subseteq Z\).

Now suppose that \(\sigma^{-1}(Z) \subseteq Z\). Fix an element \(w \in \sigma(Y)\), so that \(w = \sigma(y)\) with \(y \in Y \cap \operatorname{Dom}(\sigma)\). Suppose that \(w \in Z\). Notice that \(y \in \sigma^{-1}(\{w\}) \subseteq \sigma^{-1}(Z) \subseteq Z\), so \(y \in Z\). This is impossible since \(y \in Y\), and \(Y\) and \(Z\) are disjoint sets. Therefore, \(w \in Y\), and consequently, we get that \(\sigma(Y) \subseteq Y\).

To prove the second item, it is enough to show that $\sigma_{|_Y}(Y\cap \Dom(\sigma))=\Img(\sigma)\cap Y$, since $\Img(\sigma)$ is an open subset of $X$.

First, we show that \(\sigma(\sigma^{-1}(Y)) = Y \cap \operatorname{Im}(\sigma)\). The inclusion \(\sigma(\sigma^{-1}(Y)) \subseteq Y \cap \operatorname{Im}(\sigma)\) is clear, so we show that the reverse inclusion also holds. Fix an \(y \in Y \cap \operatorname{Im}(\sigma)\). Then \(y = \sigma(x)\) for some \(x \in \operatorname{Dom}(\sigma)\cap \sigma^{-1}(Y) \). Therefore, \(y \in \sigma(\sigma^{-1}(Y))\) as desired. 

Now we show that \(\sigma_{|_Y}(Y \cap \operatorname{Dom}(\sigma)) = \operatorname{Im}(\sigma) \cap Y\). Since \(\sigma^{-1}(Y) \subseteq Y\) and \(\sigma^{-1}(Y) \subseteq \operatorname{Dom}(\sigma)\), we have
\[ \sigma^{-1}(Y) = \sigma^{-1}(Y) \cap \operatorname{Dom}(\sigma) \subseteq Y \cap \operatorname{Dom}(\sigma). \]
By applying \(\sigma\) on both sides of the inclusion \(\sigma^{-1}(Y) \subseteq Y \cap \operatorname{Dom}(\sigma)\), we get
\[ \sigma(\sigma^{-1}(Y)) \subseteq \sigma(Y \cap \operatorname{Dom}(\sigma)). \]

Since \(\sigma(\sigma^{-1}(Y)) = Y \cap \operatorname{Im}(\sigma)\) and \(Y\) is \(\sigma\)-invariant, we have
\[ Y \cap \operatorname{Im}(\sigma) = \sigma(\sigma^{-1}(Y)) \subseteq \sigma(Y \cap \operatorname{Dom}(\sigma)) = \sigma(Y \cap \operatorname{Dom}(\sigma)) \cap \operatorname{Im}(\sigma) \subseteq Y \cap \operatorname{Im}(\sigma), \]
and consequently, we get that \(\sigma(Y \cap \operatorname{Dom}(\sigma)) = Y \cap \operatorname{Im}(\sigma)\). Since \(\sigma(Y \cap \operatorname{Dom}(\sigma)) = \sigma_{|_Y}(Y \cap \operatorname{Dom}(\sigma))\), it follows that \(\sigma_{|_Y}(Y \cap \operatorname{Dom}(\sigma)) = Y \cap \operatorname{Im}(\sigma)\).
\end{proof}

Before we prove that the restriction of a DR system to a $(\sigma, \sigma^{-1})$-invariant set is again a DR system, we make precise the definition of restriction. 

\begin{definition}
Given a DR system $(X,\sigma)$ and a subset $Y$ of $X$ such that $\sigma(Dom(\sigma)\cap Y)$ is an open set of $Y$ (in the induced topology), we denote by  $(Y, \sigma_{|_Y})$ the restricted DR system where $\sigma_{|_Y}$ is the restriction of $\sigma$ to $Y$, $\Dom(\sigma_{|_Y})=\Dom(\sigma)\cap Y$, and $\Img(\sigma_{|_Y})=\sigma(Dom(\sigma)\cap Y)$.
    
\end{definition}

\begin{proposition}\label{restriction} Let $(X,\sigma)$ be a DR system. If $Y\subseteq X$ is a $(\sigma, \sigma^{-1})$-invariant set, then $(Y, \sigma_{|_Y})$ is a DR system.
\end{proposition}

\begin{proof}

Clearly, \(Y \cap \operatorname{Dom}(\sigma)\) is an open subset of \(Y\) (with the induced topology), since \(\operatorname{Dom}(\sigma)\) is an open subset of \(X\). Since \(Y\) is \(\sigma\)-invariant, \(\sigma(Y \cap \operatorname{Dom}(\sigma))\) is a subset of \(Y\). Given that \(Y\) is \((\sigma, \sigma^{-1})\)-invariant, from the second item of Lemma~\ref{lema_rest}, we obtain that \(\sigma(Y \cap \operatorname{Dom}(\sigma))=\operatorname{Im}(\sigma) \cap Y\) is an open set in \(Y\). Therefore, \((Y, \sigma_{|_Y})\) is a restricted DR system.
\end{proof}

We end this section with a lemma that will be useful throughout this work.

\begin{lemma}\label{subset lemma}
    Let $(X,\sigma)$ be a DR system, and let $U\subseteq X$.
    \begin{enumerate}
    \item For $k\leq 0$ we have $\sigma^{-k}(\sigma^k(U))=U\cap \Img(\sigma^{-k})$, and for $k\geq 1$ we have $\sigma^{-k}(\sigma^k(U))\subseteq \Dom(\sigma^k)$. Consequently, for each $k\in \Z$, we have that $\sigma^{-k}(\sigma^k(U))\subseteq \Dom(\sigma)\cup U$.
    \item     
For each \(n \geq 1\) and \(k \in \mathbb{Z}\), the following holds:
\[
\sigma^n(U) \cap \sigma^{-k}(\sigma^k(U)) \subseteq \operatorname{Dom}(\sigma) \cup U,
\]
and if \(U \subseteq \operatorname{Dom}(\sigma)\), then
\[
\sigma^n(U) \cap \sigma^{-k}(\sigma^k(U)) \subseteq \operatorname{Dom}(\sigma).
\]
    \item If $\sigma$ is injective, then $\sigma^{-k}(\sigma^k(U))=U\cap \Dom(\sigma^k)$ for each $k\geq 0$. 
    \item If $\sigma$ is injective, then $\sigma^{-k}(\sigma^k(U))\subseteq U$ for each $k\in \Z$.
    
     \end{enumerate}
\end{lemma}

\begin{proof} 

To prove the first item, consider the case \( k \geq 1 \). In this case, note that for any subset \( Y \subseteq X \), the inverse image \( \sigma^{-k}(Y) \) is contained within \( \operatorname{Dom}(\sigma^k) \). Consequently, we have
\[
\sigma^{-k}(\sigma^k(U)) \subseteq \operatorname{Dom}(\sigma^k) \subseteq \operatorname{Dom}(\sigma).
\]
For \( k \leq 0 \), observe that an element \( x \) belongs to \( \sigma^{-k}(\sigma^k(U)) \) if, and only if, \( x = \sigma^{-k}(y) \) for some \( y \in \sigma^k(U) \) such that \( \sigma^{-k}(y)  \in  U \cap \operatorname{Im}(\sigma^{-k}) \). Therefore, we have $
\sigma^{-k}(\sigma^k(U)) = U \cap \operatorname{Im}(\sigma^{-k}).$

The second item follows from the first one.

To prove the third item, observe that \( x \in \sigma^{-k}(\sigma^k(U)) \) if and only if there exists \( y \in U \cap \operatorname{Dom}(\sigma^k) \) such that \( \sigma^k(x) = \sigma^k(y)\). In the presence of injectivity of \(\sigma\), this implies that 
\[
\sigma^{-k}(\sigma^k(U)) = U \cap \operatorname{Dom}(\sigma^k).
\]

The fourth item follows from the first and third items.
\end{proof}

\subsection{The metric entropy of DR systems.}

In this section, we recall the definition of the metric entropy of a DR system, as introduced in \cite{DDF}, and introduce the notion of generating sets, which we show can also be used to define metric entropy. 

 For $n\in\NN$ and $x\in X$, using the convention that $\Dom (\sigma^0)=X$, we define 
\begin{align}\label{def:I_n}
 I_n(x)=\{i\in\{0,1,\dots, n-1\}:x\in \Dom(\sigma^i)\}.
 \end{align}
 To measure the distance between the iterates of two points, we use the map $d_n:X\times X\to[0, +\infty)$, defined by
 $$
 d_n(x,y)=\max_{i\in I_n(x)\cap I_n(y)} d(\sigma^i(x),\sigma^i(y)).
 $$

In \cite{DDF}, the metric entropy is defined in terms of separated sets and spanning sets. In our work, it will be convenient to work with a slight modification of a spanning set, which we call a generating set; see below.

\begin{definition}\label{definicaosepspan} Let $(X,\sigma)$ be a DR system, $K$ a compact subset of $X$, $n\in\N$, and $Y$ a subset of $X$. Given $\varepsilon>0$, we say that:
\begin{itemize}
 \item a subset $A\subset K\cap Y$ is an $(n,\varepsilon,\sigma, K\cap Y)$-separated set if $d_n(x,y)>\varepsilon$ for every $x,y\in A $, $x\neq y$.
 
 \item a subset $B\subset K\cap Y$ is an $(n,\varepsilon,\sigma, K\cap Y)$-spanning set if, for every $x\in K\cap Y$, 
 there exists $y\in B$ such that  $d_n(x,y)\leq\varepsilon$.
 
 \item a subset $C\subset K\cap Y$ is an $(n,\varepsilon,\sigma, K\cap Y)$-generating set if, for every $x\in K\cap Y$, 
 there exists $y\in B$ such that  $d_n(x,y)<\varepsilon$.
 
 \end{itemize}
 We use the notation $sep(n,\varepsilon, \sigma, K\cap Y)$ to denote the largest cardinality of the $(n,\varepsilon,\sigma, K\cap Y)$-separated sets, $span(n,\varepsilon, \sigma, K\cap Y)$ to denote the smallest cardinality of the $(n,\varepsilon,\sigma, K\cap Y)$-spanning sets, and $gen(n,\varepsilon, \sigma, K\cap Y)$ to denote the smallest cardinality of the $(n,\varepsilon,\sigma, K\cap Y)$-generating sets. 
 \end{definition}

Since a local homeomorphism $(X,\sigma)$ may not be defined in the whole space, in \cite{DDF} the metric entropy is defined in terms of the supremum of the cardinality of spanning/separated sets of closed sets that approximate the domain of $\sigma^{n-1}$ from inside. We use the same idea for generating sets, see below. 

\begin{definition}\label{papagaio} Let $(X,\sigma)$ be a DR system, $K$ be a compact subset of $X$, $n\in\N$, and $\varepsilon>0$. Define
\[ssep(n,\varepsilon, \sigma, K)= \displaystyle \sup_{F}sep(n,\varepsilon, \sigma, K\cap F),\] 
\[sspan(n,\varepsilon, \sigma, K)= \displaystyle \sup_{F}span(n,\varepsilon, \sigma, K\cap F),\]
and 
\[sgen(n,\varepsilon, \sigma, K)= \displaystyle \sup_{F}gen(n,\varepsilon, \sigma, K\cap F),\]
where the supremum are taken over all closed sets $F\subseteq \Dom(\sigma^{n-1}).$
\end{definition}

For metric spaces with a basis of clopen sets, we have the following characterization of the numbers defined above.

\begin{proposition}
Let $(X,\sigma)$ be a DR system, where the topology in $X$ has a basis of clopen sets, $K$ be a compact subset of $X$, $n\in\N$, and $\varepsilon>0$. Then, the numbers $ssep(n,\varepsilon, \sigma, K)$, $sspan(n,\varepsilon, \sigma, K)$, and $sgen(n,\varepsilon, \sigma, K)$ in Definition~\ref{papagaio} do not change if the supremums in their definition are altered to be taken over all clopen sets $C\subseteq \Dom(\sigma^{n-1}).$

\end{proposition}
\begin{proof}
Notice that the result follows if we show that, given a closed set $F$ contained in $\Dom(\sigma^{n-1})$, there is a clopen set $C$, contained in $\Dom(\sigma^{n-1})$, that contains $F$. We show this below.

Let $F$ be a closed set contained in $\Dom(\sigma^{n-1})$. Since $X$ is compact so is $F$. For each $y\in F$, let $V_y$ be a clopen neighborhood of $y$ contained in $\Dom(\sigma^{n-1})$ (which exists by the hypothesis on the topology of $X$). Clearly, the $V_y$ cover $F$ and hence there is a finite sub-cover, say $\{V_{y_1},\ldots, V_{y_k}\}$. Then, the union of the $V_{y_j}$, $j=1,\ldots, k$, is the desired clopen set.
\end{proof}

The entropy of a local homeomorphism $(X,\sigma)$ is defined in \cite{DDF} as follows.

\begin{definition}\label{defentropy}Let $(X,\sigma)$ be a DR system on a metric space $(X,d)$.
\begin{enumerate}
   
\item For each compact set $K\subseteq X$ and $\varepsilon>0$ define
$$
h_{\varepsilon}(\sigma,K,d)=\limsup_{n\to\infty}\frac{1}{n}\log \mbox{ssep}(n,\varepsilon,\sigma,K).$$
\item For each compact set $K\subseteq X$ define
$$
h_{d}(\sigma,K)=\lim_{\varepsilon\to0}h_{\varepsilon}(\sigma,K,d).$$
\item 
Define the metric entropy of the DR system $(X,\sigma)$ as
\begin{equation}
h_d(\sigma)=\sup_{K\subseteq X}h_{d}(\sigma,K)
\end{equation}
where the supremum is taken over all the compact sets $K\subseteq X$.
\end{enumerate}
\end{definition}

\begin{remark}\label{entropy via span}
In \cite{DDF} it is observed that the entropy can be defined by replacing $ssep(n,\varepsilon, \sigma, K)$ with $sspan(n,\varepsilon, \sigma, K)$ above. This follows from the observation, proved in \cite{DDF}, that for $K$ a compact subset of $X$, $n\in \N$ and $\varepsilon>0$ the following inequalities are true.
$$sspan(n,\varepsilon, \sigma, K)\leq ssep(n,\varepsilon, \sigma, K)\leq sspan(n,\frac{\varepsilon}{2}, \sigma, K).$$
\end{remark}

Next, we prove that similar inequalities are true for $sgen(n,\varepsilon, \sigma, K)$.

\begin{proposition}\label{span and gen}  Let $(X,\sigma)$ be a DR system, $K\subseteq X$ be a compact subset, $\varepsilon>0$ and $n\in\N$. Then 
$$sspan(n, \varepsilon, \sigma, K)\leq sgen(n, \varepsilon, \sigma, K)\leq sspan (n, \frac{\varepsilon}{2}, \sigma, K).$$
\end{proposition}
\begin{proof} Let $F\subseteq \Dom(\sigma^{n-1})$ be a closed set and $A$ an $(n,\varepsilon, \sigma, K\cap F)$-generating set. Then, in particular, $A$ is also an $(n,\varepsilon, \sigma, K\cap F)$-spanning set. Therefore, $span(n,\varepsilon, \sigma, K\cap F)\leq gen(n,\varepsilon, \sigma, K\cap F)$. Similarly, if $B$ is an $(n,\frac{\varepsilon}{2}, \sigma, K\cap F)$-spanning set then $B$ is an $(n,\varepsilon, \sigma, K\cap F)$-generating set, and so $gen(n,\varepsilon, \sigma, K\cap F)\leq span(n, \frac{\varepsilon}{2}, \sigma, K\cap F)$. Thus, we have proved that
$$span(n,\varepsilon, \sigma, K\cap F)\leq gen(n, \varepsilon, \sigma, K\cap F)\leq span(n, \frac{\varepsilon}{2}, \sigma, K\cap F).$$ The desired result follows by taking the supremum over all closed sets $F\subseteq \Dom(\sigma^{n-1})$.
\end{proof}

\begin{remark}\label{entropy via gen} It follows from the proposition above that we can replace $ssep(n,\varepsilon, \sigma, K)$ by $sgen(n, \varepsilon, \sigma, K)$ in the definition of entropy (Definition~\ref{defentropy}). 
\end{remark}

\section{Nonwandering sets}

The definition and properties of the nonwandering set of a local homeomorphism are the focus of this section. As mentioned in the introduction, when defining the nonwandering set, we consider the fact that a local homeomorphism is not necessarily invertible. Consequently, the preimage of the direct image of a set may, a priori, be larger than the original set. The precise definitions of wandering and nonwandering points are provided below.

\begin{definition}\label{wandering} Let $(X,\sigma)$ be a DR system. We say that 
$x$ is wandering if there exists an open neighborhood $U$ of $x$  such that $\sigma^n(U)\cap \sigma^{-k} (\sigma^k  (U))= \emptyset$ for all $n\geq 1$, $k\in \Z$. We denote the set of wandering points by $\W$. We say that a point $x\in X$ is nonwandering if it is not wandering, that is, if for every open neighborhood $U$ of $x$ there exists an $n\geq 1$ and a $k\in \Z$ such that $\sigma^n(U)\cap \sigma^{-k} (\sigma^k  (U))\neq \emptyset$. We denote the set of nonwandering points by $\Omega$.
\end{definition}

\begin{remark}\label{fixed points are nonwandering}
    If \( x \in \Dom(\sigma) \) is such that there exists some \( m, n \in \mathbb{N}\cup \{0\} \), with \(m\neq n\), and such that \( \sigma^m(x) = \sigma^n(x) \), then \( x \in \Omega \). In particular, every fixed point of \( \Dom(\sigma) \) is an element of \( \Omega \). To see this, let \(n<m\) be such that \(\sigma^n(x)=\sigma^m(x)\), and write $m=n+r$. Notice that, then \(\sigma^r(x)\in \sigma^{-n}(\{\sigma^n(x)\})\), since \(\sigma^n(\sigma^r(x))=\sigma^m(x)=\sigma^n(x)\). Therefore, $\sigma^r(x)\in \sigma^r(U)\cap \sigma^{-n}(\sigma^n(U))$ for each open set $U$ which contains $x$, and consequently \(x\in \Omega\).
\end{remark}

Next, we present several examples of DR systems and their corresponding wandering sets. Later in this section, we will characterize the wandering points in shift spaces over arbitrary alphabets. We begin with an example involving a homeomorphism of $\R$, for which our definition of wandering sets coincides with the classical one.

\begin{example}\label{entropias diferentes}
   Let \((\mathbb{R}, \sigma)\) be a DR system (with the usual metric \(d\)), where \(\sigma: \mathbb{R} \rightarrow \mathbb{R}\) is the homeomorphism \(\sigma(x) = 2x\). By Remark \ref{fixed points are nonwandering}, since \(0\) is a fixed point, we have \(0 \in \Omega\). Moreover, for each element \(x > 0\), let \(a \in \left(\frac{x}{2}, x\right)\), so that \(x \in (a, 2a)\). Define \(U = (a, 2a)\), which is an open neighborhood of \(x\), and observe that \(\sigma^n(U) \cap U = \emptyset\) for each \(n \geq 1\). Since \(\sigma\) is a homeomorphism, we have \(\sigma^{-k}(\sigma^k(U)) = U\) for each \(k \in \mathbb{Z}\). Therefore, \(\sigma^n(U) \cap \sigma^{-k}(\sigma^k(U)) = \emptyset\) for each \(n \geq 1\) and \(k \in \mathbb{Z}\). Consequently, \((0, \infty) \subseteq \W\). Similarly, one shows that \((-\infty, 0) \subseteq \W\), and so we conclude that \(\Omega = \{0\}\) and \(\W = \mathbb{R} \setminus \{0\}\).

\end{example}

In the following example, we describe the wandering sets of a DR system $(X,\sigma)$ where the domain is not the entire set $X$.

\begin{example}\label{exemplo_fecho_de_W}
Let \((X, \sigma)\) be the DR system where \(X = [0,1)\), \(\sigma\) is given by \(\sigma(x) = 2x\), and \(\Dom(\sigma) = [0, \frac{1}{2})\). By induction, we obtain that \(\Dom(\sigma^n) = [0, \frac{1}{2^n})\), \(\Img(\sigma^n) = [0,1)\), and \(\sigma^n(x) = 2^n x\) for all \(n \geq 0\). Fix \(x \in (0,1)\). Considering \(U = (a, b)\) as a neighborhood of \(x\) with \(b < 2a\), we have \(\sigma^n(U) \cap \sigma^{-k}(\sigma^k(U)) = \emptyset\) for all \(n \geq 1\) and \(k \in \mathbb{Z}\). (To see this, notice that by the fourth item of Lemma \ref{subset lemma}, \(\sigma^{-k}(\sigma^k(U)) \subseteq U\); moreover, since \(b < 2a\), \(\sigma^n(U) \cap U = \emptyset\).) Therefore, the set \((0,1)\) is a subset of \(\W\). The point \(x = 0\) is nonwandering since it is a fixed point (see Remark~\ref{fixed points are nonwandering}). Thus, we conclude that \(\W = (0,1)\) and \(\Omega = \{0\}\).
\end{example}

A similar example, defined on the Cantor set, is given below.

\begin{example}\label{Cantor}
Let \(X\) be the Cantor set on the real line, and let \(\sigma\) be the local homeomorphism given by multiplication by 3, with \(\Dom(\sigma) = [0, \frac{1}{3}] \cap X\). Notice that \(\Dom(\sigma)\) is clopen, and similarly to Example~\ref{exemplo_fecho_de_W}, one can show that \(\Omega = \{0\}\). 
\end{example}

\begin{example}\label{Sierp}
We briefly recall the construction of the Sierpinski triangle via Iterated Function System (IFS). Start with a filled equilateral triangle $S(0)$. Divide it into four smaller equilateral triangles using the midpoints of the sides. Remove the interior of the central triangle, leaving only the boundaries, to form $S(1)$. Repeat this process on each of the remaining three triangles to get $S(2)$, and continue iterating to obtain $S(n)$ for higher $n$. The Sierpinski triangle $S$ is the intersection of all sets $S(n)$ as $n$ approaches infinity, resulting in a fractal pattern.

Using an IFS, $S(1)$ can be formed by scaling three copies of $S(0)$ by a factor of $r = \frac{1}{2}$ and translating two of the triangles. If the vertices of $S(0)$ are at $(0,0)$ and $(1,0)$ on $\R^2$, the translations place the triangles at $\left(\frac{1}{2}, 0\right)$ and $\left(\frac{1}{4}, \frac{\sqrt{3}}{4}\right)$.
\begin{figure}[h!]
\includegraphics[width=17cm]{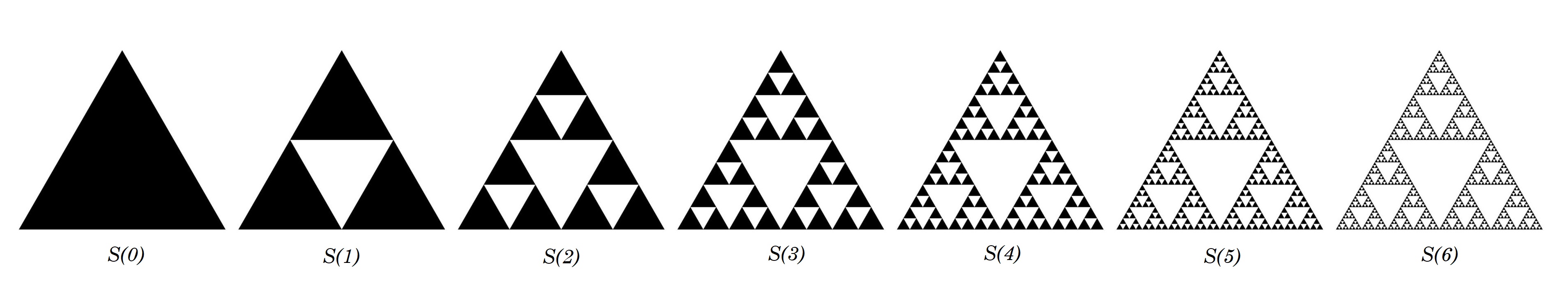}
\end{figure}

The IFS is defined by three transformations:
$$ f_1(x) = \frac{1}{2}x $$
$$f_2(x) = \frac{1}{2}x + \left[\frac{1}{2}, 0\right]$$
$$f_3(x) = \frac{1}{2}x + \left[\frac{1}{4}, \frac{\sqrt{3}}{4}\right]$$
Applying this IFS repeatedly on any compact set will result in convergence to the Sierpinski triangle, the fractal attractor of the system.

Consider the local homeomorphism $\sigma$ on $S$, defined by $\sigma(x) = 2x$, which scales points by a factor of 2. Let $\text{Dom}(\sigma) = f_1(S)$ (the image of the function $f_1$ applied to the Sierpinski triangle $S$). In this case, the image of $\sigma$ is $Im(\sigma) = S$, and the pair $(S, \sigma)$ defines a DR system. Once again, similarly to Example~\ref{exemplo_fecho_de_W}, we can show that  $\Omega = \{(0,0)\}$.

\end{example} 

Next, we outline a few properties of the wandering sets, which will be useful for the computation of the entropy of DR systems.

\begin{proposition}\label{omega_closed}
Let $(X,\sigma)$ be a DR system. The nonwandering set $\Omega$ is closed in $X$ and satisfies $\Omega\subseteq \overline{\Dom(\sigma)}$. 
\end{proposition}

\begin{proof}

To prove the first statement of the proposition, we show that \(\W\) is an open set. Let \(x \in \W\). Then, there exists an open set \(U \subseteq X\) containing \(x\) such that for all \(n \geq 1\) and \(k \in \mathbb{Z}\), it holds that \(\sigma^n(U) \cap \sigma^{-k}(\sigma^k(U)) = \emptyset\). Notice that each element of \(U\) is also wandering, so \(U \subseteq \W\). Consequently, \(\W\) is an open set, and therefore \(\Omega\) is closed.

For the second part, if \(X = \overline{\Dom(\sigma)}\), then the result is clear. If \(X \neq \overline{\Dom(\sigma)}\), fix \(x \in X \backslash \overline{\Dom(\sigma)}\). Since \(X \backslash \overline{\Dom(\sigma)}\) is open, there exists an open set \(U \subseteq X \backslash \overline{\Dom(\sigma)}\) with \(x \in U\). In this case, since \(U \cap \Dom(\sigma) = \emptyset\), we have \(\sigma^n(U) = \emptyset\) for all \(n \geq 1\). Hence, \(\sigma^n(U) \cap \sigma^{-k}(\sigma^k(U)) = \emptyset\) for all \(n \geq 1\) and \(k \in \mathbb{Z}\). Therefore, \(x \in \W\). This implies that \(X \backslash \overline{\Dom(\sigma)} \subseteq \W\), i.e., \(\Omega \subseteq \overline{\Dom(\sigma)}\).
\end{proof}

In Section~\ref{wanderentropy}, we relate the entropy of a DR system to the entropy of the systems obtained by restricting to the (non)wandering sets. Before doing so, we must first prove that the restriction of a DR system to the (non)wandering set also yields a DR system. To establish this, in light of Proposition~\ref{restriction}, we prove below that both the nonwandering and wandering sets are invariant under the forward and inverse images of the local homeomorphism.

\begin{proposition}\label{invariantshallbe}
Let $(X,\sigma)$ be a DR system. Then the nonwandering set $\Omega$, the wandering set $\W$, and the closure of $\W$ in $X$, denoted by $\overline{\W}$, are $(\sigma, \sigma^{-1})$-invariant sets. 
\end{proposition}

\begin{proof}

We first show that \(\sigma^{-1}(\W) \subseteq \W\). Let \(x \in \sigma^{-1}(\W)\), and let \(V'\) be an open neighborhood of \(\sigma(x)\) such that \(\sigma^{m}(V') \cap \sigma^{-p}(\sigma^p(V')) = \emptyset\) for each \(m \geq 1\) and \(p \in \mathbb{Z}\). Since \(\sigma\) is a local homeomorphism and \(\Dom(\sigma)\) is open, there exists an open neighborhood \(U\) of \(x\) with \(U \subseteq \Dom(\sigma)\) and an open neighborhood \(V\) of \(\sigma(x)\) such that \(\sigma(U) = V \subseteq V'\). In particular, \(\sigma^{m}(V) \cap \sigma^{-p}(\sigma^p(V)) = \emptyset\) for each \(m \geq 1\) and \(p \in \mathbb{Z}\). We now show that \(U\) satisfies the wandering condition for \(x\).

Suppose that there exist \(n \geq 1\) and \(k \in \mathbb{Z}\) such that \(\sigma^n(U) \cap \sigma^{-k}(\sigma^k(U)) \neq \emptyset\). Let \(z \in \sigma^n(U) \cap \sigma^{-k}(\sigma^k(U))\). From the second item of Lemma \ref{subset lemma}, we have \(z \in \Dom(\sigma)\). Then \(\sigma(z) \in \sigma^{n+1}(U)\) (since \(z \in \sigma^n(U)\)) and \(\sigma(z) \in \sigma^{-k+1}(\sigma^k(U))\), as \(z \in \sigma^{-k}(\sigma^k(U))\)). Since \(V = \sigma(U)\), we get
\[
\sigma(z) \in \sigma^{n+1}(U) \cap \sigma^{-k+1}(\sigma^k(U)) = \sigma^n(V) \cap \sigma^{-k+1}(\sigma^{k-1}(V)),
\]
so \(\sigma^n(V) \cap \sigma^{-k+1}(\sigma^{k-1}(V)) \neq \emptyset\). This leads to a contradiction, since \(V \subseteq V'\) and \(\sigma^m(V') \cap \sigma^{-p}(\sigma^p(V')) = \emptyset\) for each \(m \geq 1\) and \(p \in \mathbb{Z}\). Therefore, we conclude that \(\sigma^{-1}(\W) \subseteq \W\).

Next, we show that \(\sigma(\W) \subseteq \W\).

Let \(x \in \W\). Choose \(U\) as an open neighborhood of \(x\) that satisfies the wandering condition for \(x\) and such that \(\sigma|_U\) is a homeomorphism. We will show that \(\sigma(U)\) is an open neighborhood of \(\sigma(x)\) that satisfies the wandering condition for \(\sigma(x)\).

Suppose there exist \(n \geq 1\) and \(k \in \mathbb{Z}\) such that 
\[
\sigma^n(\sigma(U)) \cap \sigma^{-k}(\sigma^k(\sigma(U))) \neq \emptyset.
\]

If \(k \geq 0\), then by applying \(\sigma^k\) to a point in the intersection above and commuting \(\sigma^n\) with \(\sigma\), we obtain that \(\sigma^{k+1}(\sigma^n(U)) \cap \sigma^{k+1}(U) \neq \emptyset\). This implies that \(\sigma^n(U) \cap \sigma^{-k-1}(\sigma^{k+1}(U)) \neq \emptyset\), which is a contradiction.

We are left with the case \(k < 0\). For such \(k\), note that \(\sigma^{-k}(\sigma^k(\sigma(U))) \subseteq \sigma(U)\). Since we know that \(\sigma^n(\sigma(U)) \cap \sigma^{-k}(\sigma^k(\sigma(U))) \neq \emptyset\), the previous observation allows us to conclude that \(\sigma^n(\sigma(U)) \cap \sigma(U) \neq \emptyset\). Hence, \(\sigma(\sigma^n(U)) \cap \sigma(U) \neq \emptyset\), and therefore \(\sigma^n(U) \cap \sigma^{-1}(\sigma(U)) \neq \emptyset\), which is a contradiction.

Therefore, \(\sigma(x) \in \W\), and consequently \(\sigma(\W) \subseteq \W\).

Since \(\W\) is \(\sigma\)-invariant and \(\sigma^{-1}(\W) \subseteq \W\), it follows from Item 1 of Lemma~\ref{lema_rest} that \(\Omega\) is \(\sigma\)-invariant and \(\sigma^{-1}(\Omega) \subseteq \Omega\).

To complete the proof of the proposition, we show that \(\overline{\W}\) is \((\sigma, \sigma^{-1})\)-invariant.

We begin by proving that \(\overline{\W}\) is \(\sigma\)-invariant. Let \(x \in \overline{\W} \cap \Dom(\sigma)\), and let \(U\) be an open neighborhood of \(\sigma(x)\). Then \(\sigma^{-1}(U) \cap \Dom(\sigma)\) is an open neighborhood of \(x\). Since \(x \in \overline{\W}\), there exists \(w \in \W \cap (\sigma^{-1}(U) \cap \Dom(\sigma))\). Hence, as \(\W\) is \(\sigma\)-invariant, we obtain that \(\sigma(w) \in \W \cap U\), and thus \(\sigma(x) \in \overline{\W}\).

To show that \(\overline{\W}\) is \(\sigma^{-1}\)-invariant, let \(x \in \sigma^{-1}(\overline{\W})\) and let \(U\) be an open neighborhood of \(x\). Without loss of generality, we may assume that \(U \subseteq \Dom(\sigma)\) and that \(\sigma|_U\) is a homeomorphism. Then \(\sigma(U)\) is an open neighborhood of \(\sigma(x)\). Since \(\sigma(x) \in \overline{\W}\), there exists \(y \in \W \cap \sigma(U)\). Let \(u \in U\) be such that \(\sigma(u) = y\). As \(\W\) is \(\sigma^{-1}\)-invariant, we conclude that \(u \in \W\) and hence \(x \in \overline{\W}\).
\end{proof}

\begin{corollary}\label{restricted systems}
    Let $(X,\sigma)$ be a DR system. Then the restricted systems $(\Omega,\sigma_{|_\Omega} )$, $(\W, \sigma_{|_\W})$, and $(\overline{\W}, \sigma_{|_{\overline{\W}}} )$ are all DR systems.
\end{corollary}

\begin{proof} This follows directly from Propositions~\ref{invariantshallbe} and \ref{restriction}.
\end{proof}

\subsection{Nonwandering sets in subshifts over countable alphabets}\label{shifts}

In this section, we characterize wandering and non-wandering points in subshifts over a countable alphabet. Due to the close relationship between DR systems and C*-algebras, we adopt the definition of subshifts introduced by Ott, Tomforde, and Willis \cite{OTW}, which has been further explored in works such as \cite{BGGV3, BGGV, gonccalves2024socle,GSS2s}. We refer to these as OTW-subshifts, noting that they coincide with the traditional notion of shift spaces when the alphabet is finite. For the reader's convenience, we briefly recall the definition of the OTW-subshift space and its associated DR system below. For more details, see \cite{BGGV, OTW}.

Let $\alf$ be a countable alphabet. Define $\tilde{\alf}$ as 
\[
\tilde{\alf} := \begin{cases}
\alf & \text{if } \alf \text{ is finite,} \\
\alf \cup \{\infty\} & \text{if } \alf \text{ is infinite, where } \infty \text{ is a new symbol not in } \alf.
\end{cases}
\]
Let
\[
\Sigma_\alf := \{(x_i)_{i \in \mathbb{N}} \in \tilde{\alf}^\mathbb{N} : x_i = \infty \text{ implies } x_{i+1} = \infty\},
\]
$\Sigma_\alf^{\text{inf}} := \alf^\mathbb{N}$, and $\Sigma_\alf^{\text{fin}} := \Sigma_\alf \setminus \Sigma_\alf^{\text{inf}}$. When the alphabet is infinite, the set $\Sigma_\alf^{\text{fin}}$ is identified with all finite sequences over $\alf$ via the identification
\[
(x_0 x_1 \dots x_k \infty \infty \infty \dots) \equiv (x_0 x_1 \ldots x_k).
\]
The sequence $(\infty \infty \infty \ldots)$ is denoted by $w$ and is identified with the empty word in $\alf^0$. A block of an infinite sequence $x = (x_i)_{i \in \mathbb{N}} \in \alf^{\mathbb{N}}$ is a finite string $x_k \ldots x_j$ contained in $x$. Let $F \subseteq \alf^* = \bigcup_{k=0}^{\infty} \alf^k$. We define
\begin{align*}
\osf^{\text{inf}}_F &:= \{ x \in \alf^\mathbb{N} : \text{no block of } x \text{ is in } F \}, \\
\osf^{\text{fin}}_F &:= \{ x \in \Sigma_\alf^{\text{fin}} : \text{there are infinitely many } a \in \alf \text{ for which } \\
& \qquad \qquad \quad \ \text{there exists } y \in \alf^\mathbb{N} \text{ such that } xay \in \osf^{\text{inf}}_F \};\\
\osf^{\text{fin}^*}_F &:= \osf^{\text{fin}}_F \setminus \{w\}.
\end{align*}

The OTW-subshift associated with $F$ is defined as $\osf_F^{\text{OTW}} := \osf^{\text{inf}}_F \cup \osf^{\text{fin}}_F$ together with the shift map $\sigma: \osf_F^{\text{OTW}} \to \osf_F^{\text{OTW}}$, where the shift map $\sigma$ is defined by
\[
\sigma(x) = \begin{cases}
x_1x_2\ldots, & \text{if } x = x_0x_1x_2\ldots \in \osf_F^{\text{inf}}, \\
x_1\ldots x_{n-1}, & \text{if } x = x_0\ldots x_{n-1} \in \osf_F^{\text{fin}} \text{ and } n \geq 2, \\
w, & \text{if } x = x_0 \in \osf_F^{\text{fin}}\cup\{w\}.
\end{cases}
\]

When \( F = \emptyset \), the OTW-subshift \(\osf_F^{\text{OTW}}\) equals \(\Sigma_\alf\), which we refer to as the full shift. For simplicity, when the context is clear, we denote \(\osf_F^{\text{OTW}}\) simply as \(\osf^{\text{OTW}}\).

Note that the shift map \(\sigma\) is continuous everywhere except possibly at \( w \). Therefore, it is natural to consider the shift map as a partially defined map, where \(\operatorname{Dom}(\sigma)\) consists of all elements of \(\osf^{\text{OTW}}\) with length greater than zero; that is, \(\operatorname{Dom}(\sigma) = \osf^{\text{inf}}\cup\osf^{\text{fin}^*}\). In the case of the full shift, \((\Sigma_\alf, \sigma)\) is a DR system. It remains an open question to characterize the subshifts \(\osf^{\text{OTW}}\) for which \((\osf^{\text{OTW}}, \sigma)\) forms a DR system.

To conclude this brief introduction to OTW-subshifts, we recall that the topology on \(\osf^{\text{OTW}}\) has a basis consisting of generalized cylinders, defined as follows.

For a finite set \( F \subseteq \alf \) and \( \alpha \in \cup_{k=0}^\infty \alf^k\) the associated generalized cylinder set is defined as
\[
\CZ(\alpha, F) := \{y \in \osf^{\text{OTW}} : y_i = \alpha_i \ \forall \, 1 \leq i \leq |\alpha|, \ y_{|\alpha|+1} \notin F\}.
\]

\begin{remark}
OTW-subshifts can also be characterized as closed, invariant subsets that satisfy the infinite extension property. Moreover, OTW-subshifts are metric spaces, see \cite{OTW}.
\end{remark}

Next, we characterize nonwandering and wandering elements in OTW-subshifts that are DR systems. To do so, we first define irrational and rational sequences. This terminology is inspired by the one used in \cite{Irreduciblerepresentations} for irrational paths.

\begin{definition}
    Let \((X, \sigma)\) be an OTW shift space over an arbitrary alphabet. An infinite sequence \(x \in X\) is called rational (or eventually periodic) if there exist finite words \(\alpha\) and \(\beta\) with \(|\alpha| \geq 0\) and \(|\beta| \geq 1\) such that \(x = \alpha \beta^\infty\). An infinite sequence that is not rational is called irrational.
\end{definition}

\begin{proposition}\label{example_shift_spaces}
    Let \((X, \sigma)\) be an OTW subshift space over a countable alphabet \(A\), such that \((X, \sigma)\) is a DR system. Then the following hold:
    \begin{enumerate}
        \item\label{item1ex} Every rational element of \(X\) belongs to \(\Omega\).
        
        \item\label{item2ex} Let \(z \in X\) be such that every open neighborhood of \(z\) contains a rational element. Then \(z \in \Omega\).
        
        \item\label{item3ex} Every irrational infinite sequence in \(X\) that is an isolated point of \(X\) belongs to \(\W\).
        
        \item\label{item4ex} Either all elements of finite length belong to \(\Omega\), or all elements of finite length belong to \(\W\). Moreover, it suffices to determine the classification of the empty word to determine the behavior of all finite-length elements.
    \end{enumerate}
\end{proposition}

\begin{proof}      
   
To prove the first item, let \( x \in X \) be a rational element, i.e., \( x = ab^\infty \), where \( a \) and \( b \) are finite words. Let \( m = |b| \). Since \( \sigma^m(b^\infty) = b^\infty \), it follows from Remark~\ref{fixed points are nonwandering} that \( b^\infty \in \Omega \). By Proposition~\ref{invariantshallbe}, \(\Omega\) is \(\sigma^{-1}\)-invariant, and therefore \( x \in \Omega \).

The second item follows directly from the first item and from the fact that \(\Omega\) is closed (see Proposition~\ref{omega_closed}).

To prove the third item, suppose that \( x \) is an irrational element. Let \( U = \{x\} \), and let \( n \in \mathbb{N} \) and \( k \in \mathbb{Z} \). Write \( x = a y \) where \( |a| = n \), so that \( \sigma^n(x) = y \). If we assume that \( \sigma^n(U) \cap \sigma^{-k}(\sigma^k(U)) \neq \emptyset \), then \( y \in \sigma^{-k}(\sigma^k(\{x\})) \).

For \( k \leq 0 \), by the first item of Lemma~\ref{subset lemma}, we have \( \sigma^{-k}(\sigma^k(\{x\})) \subseteq \{x\} \), implying that \( y = x \). Thus, \( y = x = ay \), which means \( y = a^\infty \), a contradiction since \( x \) is irrational.

For \( k \geq 1 \), \( y \in \sigma^{-k}(\sigma^k(\{x\})) \) implies that \( \sigma^k(y) = \sigma^k(x) \), and consequently \( y = bz \) and \( x = cz \), where \( |b| = k = |c| \). Therefore, \( cz = x = ay = abz \). Since \( |ab| > |c| \), the equality \( cz = abz \) implies \( z = dz \) where \( |d| \geq 1 \). From \( z = dz \), we get \( z = d^\infty \), which is impossible since \( x = cz \) is irrational. Thus, \( x \in \W \).

The last item follows from the fact that \(\Omega\) and \(\W\) are \(\sigma^{-1}\)-invariant (see Proposition~\ref{invariantshallbe}).

\end{proof}

We conclude this section with several examples that apply the result above to illustrate possible wandering and non-wandering sets. In particular, we present examples where the empty word is wandering and where it is non-wandering. Additionally, we provide examples to demonstrate that the converses of Items 2 and 3 of the proposition above are not necessarily true.

\begin{example} 
In the full shift $(X, \sigma)$ over an arbitrary alphabet $A$, which is a DR system, every point is nonwandering. This follows from the second item of Proposition~\ref{example_shift_spaces}.
\end{example}

\begin{example}\label{OTW example} 
Consider the following subshift of the OTW-full shift over $\mathbb{N}$. Define the infinite sequence $p = 1234\ldots$, and let $Y = \{\sigma^k(p) : k \in \mathbb{N} \cup \{0\}\}$. Now define $X = Y \cup \{w\}$, where $w$ is the empty sequence. Notice that $X$ is a subshift of $\Sigma_{\mathbb{N}}$, and that $(X, \sigma)$ is a DR system.

We compute the wandering set $\W$ and the non-wandering set $\Omega$ of $(X, \sigma)$.

To see that $w$ is nonwandering, let $U$ be an open set in $X$ containing $w$. Since $U$ is an infinite set, there exist $x, y \in Y \cap U$ such that $y = \sigma^n(x)$ for some $n \in \mathbb{N}$. Notice that $\sigma^{-0}(\sigma^0(U)) = U$, so $y \in \sigma^n(U) \cap \sigma^{-0}(\sigma^0(U))$, and hence $\sigma^n(U) \cap \sigma^{-0}(\sigma^0(U)) \neq \emptyset$. Therefore, $w \in \Omega$.

By Item 3 of Proposition~\ref{example_shift_spaces}, all elements of $Y$ are wandering.

Thus, we obtain that $\Omega = \{w\}$ and $\W = Y$.

\end{example}

\begin{remark}
    In Proposition~\ref{omega_closed}, we proved that the nonwandering set $\Omega$ is always closed. The above example shows that it is not necessarily open (hence the wandering set $\W$ is not always closed). 
\end{remark}

\begin{example}\label{contraexitem2}
Let $\alf = \mathbb{N}$ and 
\[
\osf = \left\{ (x_i) : x_{i+1} = x_i + 1 \right\} \cup \left\{ (x_i) : \exists! \, j \text{ such that } x_{j+1} = 2x_j + 1 \text{ and } x_{i+1} = x_i + 1 \, \forall i \neq j \right\}.
\]
Consider the associated OTW-subshift, which is $\osf^{OTW}= \osf \cup \{ w \}$.
Notice that $(\osf^{OTW},\sigma)$ is a DR system. Moreover, $X$ can be seen as the set of infinite paths in the labeled graph below. 

\begin{center}
\hspace{-6cm}\setlength{\unitlength}{2mm}
	\begin{picture}(30,25)
		\put(2,0){\circle*{0.7}}
		\put(12,0){\circle*{0.7}}
		\put(22,0){\circle*{0.7}}
		\put(32,0){\circle*{0.7}}
		\put(42,0){\circle*{0.7}}
		\put(9,7){\circle*{0.7}}
		\put(19,7){\circle*{0.7}}
		\put(29,7){\circle*{0.7}}
		\put(39,7){\circle*{0.7}}
            \put(16,14){\circle*{0.7}}
		\put(26,14){\circle*{0.7}}
		\put(36,14){\circle*{0.7}}
		\put(46,14){\circle*{0.7}}
            \put(23,21){\circle*{0.7}}
		\put(33,21){\circle*{0.7}}
		\put(43,21){\circle*{0.7}}
		\put(53,21){\circle*{0.7}}

      	\put(45,0){$\ldots $}
    	\put(50,10){$\ldots $}
		\put(24,22){\reflectbox{$\ddots$}}
		\put(34,22){\reflectbox{$\ddots$}}
		\put(44,22){\reflectbox{$\ddots$}}
		\put(54,22){\reflectbox{$\ddots$}}

		\put(2,0){\vector(1,0){10}}	
		\put(12,0){\vector(1,0){10}}	
		\put(22,0){\vector(1,0){10}}	
		\put(32,0){\vector(1,0){10}}	

            \put(2,0){\vector(1,1){7}}	
		\put(12,0){\vector(1,1){7}}	
		\put(22,0){\vector(1,1){7}}	
		\put(32,0){\vector(1,1){7}}
    
            \put(9,7){\vector(1,1){7}}	
		\put(19,7){\vector(1,1){7}}	
		\put(29,7){\vector(1,1){7}}	
		\put(39,7){\vector(1,1){7}}

            \put(16,14){\vector(1,1){7}}	
		\put(26,14){\vector(1,1){7}}	
		\put(36,14){\vector(1,1){7}}	
		\put(46,14){\vector(1,1){7}}
 
		\put(7,-1.8){\normalsize$1$}
		\put(17,-1.8){\normalsize$2$}
		\put(27,-1.8){\normalsize$3$}
		\put(37,-1.8){\normalsize$4$}
		\put(4,4){\normalsize$1$}
		\put(14,4){\normalsize$2$}
		\put(24,4){\normalsize$3$}
		\put(34,4){\normalsize$4$}
            \put(11,11){\normalsize$3$}
		\put(21,11){\normalsize$5$}
		\put(31,11){\normalsize$7$}
		\put(41,11){\normalsize$9$}
            \put(18,18){\normalsize$4$}
		\put(28,18){\normalsize$6$}
		\put(38,18){\normalsize$8$}
		\put(48,18){\normalsize$10$}
		
		\put(16,-4.2){\normalsize  Labeled graph $\mathcal{E}$}
\end{picture}
\end{center}
\vspace{1cm}

We first establish that the element $x = 12345\ldots \in \textsf{X}$ is nonwandering. To prove this, for every $n \geq 1$, consider $V_n = Z[123\ldots n]$. Then, the sequences $(n+j)_{j=1}^{\infty}$ and $(2n+j)_{j=1}^{\infty}$ belong to $\sigma^n(V_n)$. Therefore, $(n+j)_{j=1}^{\infty} \in \sigma^n(V_n) \cap \sigma^{-n}(\sigma^n(V_n))$. Since every open neighborhood of $x$ contains cylinders of the form $V_n$, we conclude that $x \in \Omega$.

Now, notice that every element of $\osf$ is in the extended orbit of $x$, that is, every element of $\osf$ belongs to $\bigcup_{k,n = 0,1,\ldots} \sigma^{-n}(\sigma^k(x))$. Since by Proposition~\ref{invariantshallbe}, $\Omega$ is $(\sigma,\sigma^{-1})$-invariant, and $x \in \Omega$, we conclude that $\osf$ is contained in $\Omega$. Finally, as $\Omega$ is closed, we obtain that $w \in \Omega$. Hence, $\textsf{X} = \Omega$.

\end{example}

\begin{remark}
Observe that all infinite elements in Example~\ref{contraexitem2} are irrational. Moreover, this example shows that even for infinite irrational elements, the converse of Item~\ref{item2ex} in Proposition~\ref{example_shift_spaces} does not hold.
\end{remark}

\begin{example}\label{emptywandering}
    
    Let $\alf = \mathbb{N}$ and consider the following subshift of the OTW-full shift. Define $p = 1010^210^310^4\ldots$, and let  
\[
\osf = \{\sigma^n(p) : n = 0,1,2,\ldots \} \cup \{ 0^n 1 0^\infty : n = 0, 1, 2, \ldots \} \cup \{ 0^\infty \} \cup \{ np : n \in \mathbb{N} \} \cup \{ w \}.
\]
Notice that $\osf$ is an OTW-subshift, and $(\osf, \sigma)$ is a DR system. We now compute the wandering and nonwandering sets.

From the first item of Proposition~\ref{example_shift_spaces}, we have $\{ 0^\infty \} \cup \{ 0^n 1 0^\infty : n = 0, 1, 2, \ldots \} \subseteq \Omega$. Moreover, from the third item of the same proposition, each element of the form $\sigma^n(p)$ or $np$ belongs to the wandering set $\W$. The element $w$ also belongs to $\W$. 
To see this, consider the open set $U = \{ 2p, 3p, 4p, \ldots \} \cup \{ w \}$, which contains $w$. Notice that $\sigma^{-k}(\sigma^k(U)) = U \setminus \{ w \} \cup \{ 0p, 1p \}$ for each $k \geq 0$, and $\sigma^{-k}(\sigma^k(U)) = \emptyset$ for each $k < 0$. Since $\sigma^n(U) = \{ \sigma^{n-1}(p) \}$ for each $n \geq 1$, we get that $\sigma^n(U) \cap \sigma^{-k}(\sigma^k(U)) = \emptyset$ for all $n \geq 1$ and $k \in \mathbb{Z}$. 

Therefore, the nonwandering set is $\Omega = \{ 0^\infty \} \cup \{ 0^n 1 0^\infty : n = 0, 1, 2, \ldots \}$, and the wandering set is $\W = \osf \setminus \Omega$.

\end{example}

\begin{remark}
    The example above shows that elements of finite length can be wandering.
\end{remark}

\section{Entropy and the nonwandering set}\label{wanderentropy}

In this section, we show that the entropy of a DR system is the maximum of the entropies of the restricted systems on the nonwandering set and the closure of the wandering set. Furthermore, we show that for a DR system defined over a compact space with a clopen domain, the entropy is concentrated in the nonwandering set. We begin by proving that the entropy of a restricted system is always less than or equal to that of the original system.

\begin{proposition}\label{prop:entropy-restriction}
Let $(X,\sigma)$ be a DR system, and $(Y,\sigma_{|_{Y}})$ be a restricted DR system. Then, $h_d(\sigma)\geq h_d( \sigma_{|_{Y}})$. 
\end{proposition}

\begin{proof}

Let \( K \) be a compact subset of \( Y \). Then, \( K \) is also a compact set in \( X \). Let \( F \subseteq \operatorname{Dom}(\sigma_{|_{Y}}^{n-1}) \) be a closed set, and let \( A \subseteq K \cap F \) be an \((n, \varepsilon, \sigma_{|_{Y}}, K \cap F)\)-separated set. Since \( A \) is finite, it is closed in \( \operatorname{Dom}(\sigma^{n-1}) \), and therefore \( A \) is also an \((n, \varepsilon, \sigma, K \cap A)\)-separated set. Thus, we have
\[
\operatorname{sep}(n, \varepsilon, \sigma_{|_{Y}}, K \cap F) \leq \operatorname{sep}(n, \varepsilon, \sigma, K \cap A),
\]
and since $
\operatorname{sep}(n, \varepsilon, \sigma, K \cap A) \leq \operatorname{ssep}(n, \varepsilon, \sigma, K),
$
we obtain
\[
\operatorname{ssep}(n, \varepsilon, \sigma_{|_{Y}}, K) \leq \operatorname{ssep}(n, \varepsilon, \sigma, K).
\]

Thus,
\[
h_{\varepsilon}(\sigma_{|_{Y}}, K, d) = \limsup_{n \to \infty} \frac{1}{n} \log \operatorname{ssep}(n, \varepsilon, \sigma_{|_{Y}}, K) \leq \limsup_{n \to \infty} \frac{1}{n} \log \operatorname{ssep}(n, \varepsilon, \sigma, K) = h_{\varepsilon}(\sigma, K, d).
\]

Now, taking the limit as \( \varepsilon \to 0 \), we obtain $
h_d(\sigma_{|_{Y}}, K) \leq h_d(\sigma, K)$, and from the definition of the entropy, we conclude that
\[
h_d(\sigma_{|_{Y}}) \leq h_d(\sigma).
\]
\end{proof}

If the space of a DR system can be expressed as a union of open invariant subsets, then the entropy of the entire system equals the maximum entropy among the corresponding restricted systems \cite[Proposition 5.3]{DDF}. Below, we observe that this result also holds when the space is a union of closed sets.

\begin{proposition}\label{max entropy}
    Let $(X, \sigma)$ be a DR system, and let $Y_i\subseteq X$ be closed $\sigma$-invariant subsets such that $(Y_i, \sigma_{|_{Y_i}})$ are restricted DR systems for $i\in \{1,...,k\}$, and such that $X=\bigcup\limits_{i=1}^k Y_i$. Then $h_d(\sigma)=\max\limits_{1\leq i\leq n}h_d(\sigma_{|_{Y_i}}).$
    \end{proposition}

\begin{proof} The proof follows similarly to the one for \cite[Proposition~5.3]{DDF}, and the necessary adaptations are straightforward.
\end{proof}

Applying the above proposition to the nonwandering set and the closure of the wandering set, we obtain the following result.

\begin{theorem}\label{entropias wandering e nonwandering}
Let \((X, \sigma)\) be a DR system, and let \((\Omega, \sigma_{|_{\Omega}})\) and \((\overline{\W}, \sigma_{|_{\overline{\W}}})\) be the restricted DR systems as described in Corollary \ref{restricted systems}. Then,

$$h_d(\sigma) = \max \{ h_d(\sigma_{|_{\Omega}}), h_d(\sigma_{|_{\overline{\W}}}) \}.$$
\end{theorem}

\begin{proof}
The proof follows directly from Propositions~\ref{omega_closed} and \ref{max entropy}.
\end{proof}

Our next objective is to identify conditions that ensure the entropy of a DR system is concentrated in the nonwandering set. However, this is not always the case, as we show in the following example.

\begin{example}
Recall Example~\ref{entropias diferentes}, where \((\R, \sigma)\) represents the DR system with \(\sigma: \R \rightarrow \R\) the homeomorphism \(\sigma(x) = 2x\). We observed that \(\Omega = \{0\}\) and \(\W = \R \setminus \{0\}\). Therefore, \(h_d(\sigma_{|_{\Omega}}) = 0\), while \(h_d(\sigma) = \log(2)\).
\end{example}

\begin{remark}
    The example above also shows that, in general, it is not true that $h_d(\sigma_{|\overline{\W}})\leq h_d(\Omega)$.
\end{remark}

In certain cases, the entropy of the system restricted to the nonwandering set is the same as the entropy of the system restricted to the closure of the wandering set, see below.

\begin{example} 
Let \( X = [0,1] \) with the usual topology of \(\R\), and let \(\Dom(\sigma) = [0, \frac{1}{2})\) with \(\sigma(x) = 2x\). As shown in Example~\ref{exemplo_fecho_de_W}, we have \(\W = (0,1]\) and \(\Omega = \{0\}\). Notice that \( h_d(\sigma_{|_\Omega}) = 0\). We argue that \( h_d(\sigma) = 0 \).

First, note that \(\Dom(\sigma^n) = [0, \frac{1}{2^n})\) for each \( n \in \N \). Moreover, for any \( x, y \in \Dom(\sigma^{n-1}) \), it holds that \( d_n(x,y) = |\sigma^{n-1}(x) - \sigma^{n-1}(y)| \). Consequently, for any \( x, y \in \Dom(\sigma^{n-1}) \) and \(\varepsilon > 0\), we have \( d_n(x, y) \geq \varepsilon \) if and only if \( |\sigma^{n-1}(x) - \sigma^{n-1}(y)| \geq \varepsilon \).

Now, fix an \(\varepsilon > 0\) and \( n \in \N \). Let \( m \in \N \) be such that \(\frac{1}{m} \leq \varepsilon\). Let \( B \subseteq [0,1) \) be such that \( |u - v| \geq \varepsilon \) for each \( u, v \in B \). Then, \( B \) contains at most \( m \) elements (since \([0,1) = \bigcup_{i=1}^m [\frac{i-1}{m}, \frac{i}{m})\), and each interval \([\frac{i-1}{m}, \frac{i}{m})\) contains at most one element of \( B \)). 

Let \( F \subseteq \Dom(\sigma^{n-1}) \) be a closed subset, and let \( A \) be an \((n, \varepsilon, \sigma, F)\)-separated set of maximal cardinality. Define \( B = \sigma^{n-1}(A) \), and notice that, as per the first paragraph, \( |u - v| \geq \varepsilon \) for each \( u, v \in B \). Since the cardinality of \( B \) is at most \( m \), the cardinality of \( A \) is also at most \( m \). Thus, we have \( ssep(n, \varepsilon, \sigma, [0,1]) \leq m \), leading to \( h_\varepsilon(\sigma, [0,1], d) = 0 \). Consequently, it follows that \( h_d(\sigma) = 0 \).
   
\end{example}

To prove our next theorem, which concerns the concentration of entropy on the nonwandering set, we require a lemma and need to revisit the definition of the dynamical ball (see \cite{DDF, entropypartialaction}). We start by recalling the definition of the dynamical ball.

\begin{definition}
    \label{dynamical ball}  
Let $(X, \sigma)$ be a DR system. For each $x \in X$, $n \in \N$, and $\varepsilon > 0$, define the sets

 $$U(x, n, \varepsilon) = \bigcap_{i \in I_n(x)} \sigma^{-i}\left(B(\sigma^i(x), \varepsilon)\right)$$ and 
 $$B(x, n, \varepsilon) = \{y \in X : I_n(y) = I_n(x) \text{ and } d_n(y, x) < \varepsilon\}.$$ 

\end{definition}

\begin{remark}
    The set \( U(x, n, \varepsilon) \), referred to as a "dynamical ball" in \cite[Remark~3.8]{DDF}, is an open subset of \( X \) that contains \( x \). Moreover, if $y\in (x,n,\varepsilon)$, then $I_n(y)\supseteq I_n(x)$. 

    If \( \text{Dom}(\sigma) \) is not clopen, the set \( B(x, n, \varepsilon) \) is not necessarily open. For instance, in the case of a row-finite OTW full shift over an infinite alphabet, the only element \( y \) in the full shift with \( I_n(y) \) equal to \( I_n \) of the empty word is the empty word itself. Thus, \( B(w, n, \varepsilon) = \{w\} \), which is not open.

    However, if \( \text{Dom}(\sigma) \) is clopen, then \( B(x, n, \varepsilon) \) is indeed an open set, as proven in the final item of the next lemma.
\end{remark}

\begin{lemma}
Let \((X, \sigma)\) be a DR system. Fix \(x \in X\) and \(n \in \N \), and let \(\varepsilon > 0\). Then the following holds:

\begin{enumerate}\label{lemmaopenball}
    \item The map \(d_n\), when restricted to \(B(x, n, \varepsilon)\), defines a metric.
    \item If \(x \in \Dom(\sigma^k)\) with \(k \geq n-1\), then \(B(x, n, \varepsilon) = U(x, n, \varepsilon)\), and additionally, \(B(x, n, \varepsilon) \subseteq \Dom(\sigma^{n-1})\).
    \item If \(\Dom(\sigma)\) is clopen, then \(B(x, n, \varepsilon)\) is an open subset in \(X\).
\end{enumerate}
\end{lemma}

\begin{proof} 

For the first item, it suffices to prove the triangle inequality, which holds because \( I_n(y) = I_n(x) \) for each \( y \in B(x, n, \varepsilon) \).

The second item follows directly from the definitions of \( U(x, n, \varepsilon) \) and \( B(x, n, \varepsilon) \).

We now prove the third item. Notice that since $\Dom(\sigma)$ is clopen, $\Dom(\sigma^k)$ is also clopen for each $k\in \N$. 

First, suppose that \(I_n(x) = \{0, 1, \dots, n-1\}\). In this case, \(x \in \text{Dom}(\sigma^{n-1})\), and by the second item, \(B(x, n, \varepsilon)\) is open because \(U(x, n, \varepsilon)\) is an open set.

Now, suppose \(I_n(x) = \{0, 1, \dots, k\}\) with \(k < n-1\). Under the assumption that \(\text{Dom}(\sigma)\) is clopen, it follows that \(\text{Dom}(\sigma^k) \setminus \text{Dom}(\sigma^{k+1})\) is an open set in \(X\), since \(\text{Dom}(\sigma^{k+1})\) is closed. Therefore, since
\[
B(x, n, \varepsilon) = \left(\text{Dom}(\sigma^k) \setminus \text{Dom}(\sigma^{k+1})\right) \cap U(x, n, \varepsilon),
\]
we conclude that \(B(x, n, \varepsilon)\) is open, as desired.
\end{proof}

Motivated by a similar result in the context of partial actions (see \cite{entropypartialaction}), we prove below that for DR systems over compact spaces with a clopen domain, the entropy is concentrated in the nonwandering set.

\begin{theorem}\label{entropias iguais} 
Let \((X, \sigma)\) be a DR system on a compact metric space \((X, d)\) with \(\text{Dom}(\sigma)\) clopen, and let \((\Omega, \sigma_{|_\Omega})\) be the restricted DR system. Then \(h_d(\sigma) = h_d(\sigma_{|_\Omega})\).
\end{theorem}
\begin{proof}

From Corollary~\ref{restricted systems} and Proposition~\ref{prop:entropy-restriction}, we have $h_d(\sigma_{|_\Omega})\leq h_d(\sigma)$. We prove the reverse inequality below.

Let \(n \in \mathbb{N}\) and \(\varepsilon > 0\). Define \(F = \text{Dom}(\sigma^{n-1})\), which is closed in \(X\) since \(\text{Dom}(\sigma)\) is closed and \(\sigma\) is continuous.

Let \(A \subseteq \Omega \cap F\) be an \((n, \varepsilon, \sigma, \Omega \cap F)\)-generating set of minimum cardinality. By Proposition~\ref{span and gen} and \cite[Lemma 3.7]{DDF}, $A$ is a finite set.

For each \(x \in A\), let \(B(x, n, \varepsilon)\) be the set given in Definition~\ref{dynamical ball}.

\vspace{0.5 cm}

\textbf{Claim 1:} The collection \(\{B(x,n,\varepsilon)\}_{x \in A}\) is an open cover of \(\Omega \cap F\).

\vspace{0.5 cm}

By Lemma \ref{lemmaopenball}, \(B(x,n,\varepsilon)\) is open in \(X\). To show that \(\{B(x,n,\varepsilon)\}_{x \in A}\) covers \(\Omega \cap F\), let \(y \in \Omega \cap F\). Since \(A\) is an \((n,\varepsilon,\sigma, \Omega \cap F)\)-generating set, there exists some \(z \in A\) such that \(d_n(y,z) < \varepsilon\). Given that \(F = \Dom(\sigma^{n-1})\) and \(y, z \in F\), we have \(I_n(y) = I_n(z)\). Therefore, \(y \in B(z,n,\varepsilon)\), proving Claim 1.$\square$

\hspace{0.5pc}

Now, let $U=\bigcup\limits_{x\in A}B(x,n,\varepsilon)$, which is an open set containing $F\cap \Omega$. Notice that $A$ is an $(n,\varepsilon, \sigma, U)$-generating set. 

\vspace{0.5 cm}
\textbf{Claim 2:} For each $N\in \N$, there exists $0<\beta\leq \varepsilon$ such that 
$$
\sigma^m(U(x,N,\beta))\cap \sigma^{-k}(\sigma^k(U(x,N,\beta)))=\emptyset
$$ for each $m\in\N$, $k\in \Z$, and $x\in F\setminus U$. 
\vspace{0.5 cm}

To prove this claim, suppose for contradiction that for each \( \beta > 0 \), there exist \( x \in F \setminus U \), \( i \in \mathbb{N} \), and \( j \in \mathbb{Z} \) such that
\(
\sigma^i(U(x, N, \beta)) \cap \sigma^{-j}(\sigma^j(U(x, N, \beta))) \neq \emptyset.
\) 
For each \( m \in \mathbb{N} \), choose \( x_m \in F \setminus U \), \( i_m \in \mathbb{N} \), and \( j_m \in \mathbb{Z} \) satisfying
\[
\sigma^{i_m}\big(U(x_m, N, \tfrac{1}{m})\big) \cap \sigma^{-j_m}\big(\sigma^{j_m}(U(x_m, N, \tfrac{1}{m}))\big) \neq \emptyset.
\]

Since \( F \setminus U \) is compact, there exists a subsequence \( (x_{m_k})_{k \in \mathbb{N}} \) that converges to some point \( z \in F \setminus U \).
Let \( V \subseteq X \) be an open neighborhood of \( z \). Then, there exists \( k_0 \in \mathbb{N} \) such that
\[
B\big(x_{m_{k_0}}, \tfrac{1}{m_{k_0}}\big) \subseteq V.
\]
As
\(
U\big(x_{m_{k_0}}, N, \tfrac{1}{m_{k_0}}\big) \subseteq B\big(x_{m_{k_0}}, \tfrac{1}{m_{k_0}}\big),
\)
we have
\(
U\big(x_{m_{k_0}}, N, \tfrac{1}{m_{k_0}}\big) \subseteq V.
\)

From our assumption, we obtain that
\[
\sigma^{i_{m_{k_0}}}\big(V\big) \cap \sigma^{-j_{m_{k_0}}}\big(\sigma^{j_{m_{k_0}}}(V)\big) \neq \emptyset.
\]
This implies that \( z \) is a non-wandering point, which contradicts the fact that all points in \( F \setminus U \) are wandering.
Therefore, Claim~2 holds. \( \square \)

\vspace{0.5pc}

Let $\beta$ be as in Claim 2 (for $N=n$), and let $B$ be an $(n,\beta, \sigma, F\setminus U)$-generating set of minimum cardinality. By Proposition \ref{span and gen} and \cite[Lemma 3.7]{DDF}, \(B\) is finite.

\vspace{0.5cm}

\textbf{Claim 3:} The collection \(\{B(x, n, \beta)\}_{x \in B}\) is an open cover of \(F \setminus U\).

\vspace{0.5cm}

The proof of this claim follows the same reasoning as in Claim 1 and is left to the reader. \(\square\)

\vspace{0.5cm}

Let \( V = \bigcup\limits_{x \in B} B(x, n, \beta) \). By the second item of Lemma~\ref{lemmaopenball}, we have \( B(x, n, \beta) \subseteq \Dom(\sigma^{n-1}) \) for each \( x \in B \). Similarly, \( B(x, n, \varepsilon) \subseteq \Dom(\sigma^{n-1}) \) for each \( x \in A \). Therefore, \( U \cup V \subseteq \Dom(\sigma^{n-1}) = F \). From Claims 1 and 3, we conclude that \( U \cup V = F \).

Since \( F = U \cup V \) is an open set, \( X \setminus F \) is compact. Therefore, there exists a finite set \( C \subseteq X \setminus F \) such that \( X \setminus F \subseteq \bigcup\limits_{x \in C} B(x, n, \varepsilon) \). This holds because each \( B(x, n, \varepsilon) \) is an open set.

Notice that \( X = \bigcup\limits_{x \in A \cup C} B(x, n, \varepsilon) \cup \bigcup\limits_{y \in B} B(y, n, \beta) \), so each \( z \in X \) belongs to some \( B(x, n, \varepsilon) \) with \( x \in A \cup C \), or \( z \) belongs to some \( B(y, n, \beta) \) with \( y \in B \).

Let \( A = \{u_h : h \in H\} \), \( B = \{v_i : i \in I\} \), and \( C = \{w_j : j \in J\} \), where \( H \), \( I \), and \( J \) are finite index sets. Define \( D = A \cup B \cup C \). Fix an \( l \in \mathbb{N} \). We now define a map
\[
\varphi_l: F \rightarrow (D \cup \{\star\})^l
\]
(where \(\star\) is just a symbol) by \(\varphi_l(x) = (y_0, y_1, \ldots, y_{l-1})\), where each \( y_k \) is defined as follows:

\begin{itemize}
\item To define \( y_0 \), notice that since \( x \in F \), there exists some \( u_{h_0} \in A \) such that \( x \in B(u_{h_0}, n, \varepsilon) \), in which case we define \( y_0 = u_{h_0} \), or there exists some \( v_{i_0} \in B \) such that \( x \in B(v_{i_0}, n, \beta) \), in which case we define \( y_0 = v_{i_0} \).

\item For each \( k \in \{1, \ldots, l-1\} \) with \( x \notin \Dom(\sigma^{kn}) \), define \( y_k = \star \).

\item For each \( k \in \{1, \ldots, l-1\} \) with \( x \in \Dom(\sigma^{kn}) \):
  \begin{itemize}
  \item If \( \sigma^{kn}(x) \in F \), then there exists some \( u_{h_k} \in A \) such that \( \sigma^{kn}(x) \in B(u_{h_k}, n, \varepsilon) \), in which case we define \( y_k = u_{h_k} \), or there exists some \( v_{i_k} \in B \) such that \( \sigma^{kn}(x) \in B(v_{i_k}, n, \beta) \), in which case we define \( y_k = v_{i_k} \).
  \item If \( \sigma^{kn}(x) \notin F \), then there exists some \( w_{j_k} \in C \) such that \( \sigma^{kn}(x) \in B(w_{j_k}, n, \varepsilon) \), in which case we define \( y_k = w_{j_k} \).
  \end{itemize}
\end{itemize}

Notice that the map \(\varphi_l\) may not be unique, as each \( y_k \) depends on the choice of \( u_h \), \( v_i \), and \( w_j \).

\vspace{0.5 cm}
\textbf{Claim 4:} For each \( x \in F \), any \( y \in B \) appears at most once in \(\varphi_l(x)\).
\vspace{0.5 cm}

Suppose that there exists an \( x \in F \) such that some \( y \in B \) appears twice in \(\varphi_l(x) = (y_0, y_1, \ldots, y_{l-1})\), and assume \( y_p = y = y_q \) with \( p < q \). Notice that in this case, \( \sigma^{pn}(x) \in B(y, n, \beta) = U(y, n, \beta) \), where the equality follows from the second item of Lemma~\ref{lemmaopenball}. Similarly, we can show that \( \sigma^{qn}(x) \in U(y, n, \beta) \).

Now observe that \( \sigma^{qn-pn}(\sigma^{pn}(x)) = \sigma^{qn}(x) \), so \( \sigma^{qn}(x) \in \sigma^{qn-pn}(U(y, n, \beta)) \cap U(y, n, \beta) \). However, by Claim 2, this is impossible. This contradiction proves Claim 4. \(\square\)

\vspace{0.5 cm}

\textbf{Claim 5:} For each \( x \in F \), at most one element of \( C \) appears in \(\varphi_l(x)\).

\vspace{0.5 cm}

To see that this claim is true, let \( x \in F \) and suppose that some element of \( C \) appears in \(\varphi_l(x) = (y_0, y_1, \ldots, y_{l-1})\). Let \( p \) be the smallest index such that \( y_p \in C \). If \( p = l-1 \), we are done. Therefore, suppose \( p < l-1 \). By the definition of \(\varphi_l\), it follows that \( x \in \Dom(\sigma^{pn}) \) and \( \sigma^{pn}(x) \notin F \). In particular, \( \sigma^{pn}(x) \notin \Dom(\sigma^n) \) (since \( \Dom(\sigma^n) \subseteq \Dom(\sigma^{n-1}) = F \)), so \( x \notin \Dom(\sigma^{(p+1)n}) \), and therefore \( y_{p+1} = \star \). By the definition of \(\varphi_l\), it follows that \( y_k = \star \) for each \( k \geq p+1 \). Hence, the unique element of \( C \) appearing in \(\varphi_l(x)\) is \( y_p \).
\(\square\)

\vspace{0.5 cm}
\textbf{Claim 6:} Let $l\in \N$, $0<m\leq nl$, and let $E$ be an $(m, 2\varepsilon, \sigma, F)$-separated set. Then $\varphi_l$ is injective on $E$.

\vspace{0.5 cm}

Let \( x, x' \in E \) such that \( \varphi_l(x) = \varphi_l(x') \), and denote \( \varphi_l(x) = (y_0, \dots, y_{l-1}) \). For any \( k \in \{0, \dots, l-1\} \), since \( \varphi_l(x) = \varphi_l(x') \), we have that \( x \in \text{Dom}(\sigma^{kn}) \) if and only if \( x' \in \text{Dom}(\sigma^{kn}) \). Furthermore, if \( x \in \text{Dom}(\sigma^{kn}) \), then \( I_n(\sigma^{kn}(x)) = I_n(\sigma^{kn}(x')) \). Assuming \( x \in \text{Dom}(\sigma^{kn}) \), by the definition of \( \varphi_l \), we obtain that \( \sigma^{kn}(x), \sigma^{kn}(x') \in B(y_k, n, \varepsilon) \) when \( y_k \in A \cup C \), or \( \sigma^{kn}(x), \sigma^{kn}(x') \in B(y_k, n, \beta) \) when \( y_k \in B \). Since \( \beta \leq \varepsilon \), it follows from the first item of Lemma~\ref{lemmaopenball} that \( d_n(\sigma^{kn}(x), \sigma^{kn}(x')) < 2\varepsilon \). Consequently, \( d(\sigma^{i+kn}(x), \sigma^{i+kn}(x')) < 2\varepsilon \) for each \( i \in I_n(\sigma^{kn}(x)) = I_n(\sigma^{kn}(x')) \).

Now, let \( 0 < m \leq nl \) and let \( p \in \{0, \dots, l-1\} \) be the greatest index such that \( pn \leq m \). We will show that \( d_m(x, x') < 2\varepsilon \).

First, suppose that \( x \in \text{Dom}(\sigma^{pn}) \). Then, in particular, \( x \in \text{Dom}(\sigma^{kn}) \) for each \( k \leq p \). From the previous discussion, we have \( d(\sigma^{i+kn}(x), \sigma^{i+kn}(x')) < 2\varepsilon \) for each \( i \in I_n(\sigma^{kn}(x)) = I_n(\sigma^{kn}(x')) \) and for each \( k \in \{0, \dots, p\} \). Since \( x \in \text{Dom}(\sigma^{pn}) \), it follows that \( I_n(\sigma^{kn}(x)) = \{0, 1, \dots, n-1\} = I_n(\sigma^{kn}(x')) \) for each \( k < p \). Therefore, we have \( d(\sigma^j(x), \sigma^j(x')) < 2\varepsilon \) for each \( j \in \{0, 1, \dots, pn-1\} \). Moreover, from the previous discussion, we have \( d(\sigma^{i+pn}(x), \sigma^{i+pn}(x')) < 2\varepsilon \) for each \( i \in I_n(\sigma^{pn}(x)) = I_n(\sigma^{pn}(x')) \). Notice that \( I_m(x) = \{0, 1, \dots, pn-1\} \cup \{i+pn : i \in I_n(\sigma^{pn}(x))\} = I_m(x') \), and hence we conclude that \( d_m(x, x') < 2\varepsilon \).

We now consider the case where \( x \notin \text{Dom}(\sigma^{pn}) \). In this setting, let \( q \in \{0, \dots, l-1\} \) be the greatest index such that \( x \in \text{Dom}(\sigma^{qn}) \) (note that \( q \) could be zero). Then, \( x \in \text{Dom}(\sigma^{kn}) \) for each \( k \leq q \). The proof that \( d_m(x, x') < 2\varepsilon \) now follows in the same way as in the previous paragraph.

Thus, we obtain that \( d_m(x, x') < 2\varepsilon \), which is a contradiction, as \( E \) is an \((m, 2\varepsilon, \sigma, F)\)-separated set. Therefore, \( x = x' \), implying that \( \varphi \) is injective on \( E \).  $\square$

\vspace{0.5 cm}
{\it {\bf Claim 7:} Let \( m \in \mathbb{N} \) be such that \( 0 < m \leq nl \). Define \( p = \#A \), \( q = \#B \), and \( r = \#C \), and let \( E \) be an \((m, 2\varepsilon, \sigma, F)\)-separated set of maximum cardinality in \( F \). Then, it follows that \( \#(\varphi_l(E)) \leq (q + 2)! \cdot \left( (q + 1)r + 1 \right) l^{q + 1}(p + 1)^l \).}
\vspace{0.5 cm}

For each \( 0 \leq j \), let \( I_j \) denote the subset of elements \( \varphi_l(x) = (y_0, \dots, y_{l-1}) \) in \( \varphi_l(E) \) with exactly \( j \) entries in \( B \cup C \). Note that for \( j > q+1 \), it follows from Claims 4 and 5 that \( I_j = \emptyset \). For each \( j \in \{0, \dots, q+1\} \), let \( M_j \subseteq I_j \) be the set of elements \( \varphi_l(x) = (y_0, \dots, y_{l-1}) \) containing no elements of \( C \), and let \( N_j \subseteq I_j \) be the set of elements \( \varphi_l(x) = (y_0, \dots, y_{l-1}) \) containing exactly one element of \( C \). From Claim 5, we have \( M_j \cup N_j = I_j \).

Observe that there are \( \frac{q!}{j!(q-j)!} \) ways to choose \( j \) elements from \( B \), and there are \( \frac{l!}{(l-j)!} \) ways to arrange each selection. Since each element \( \varphi_l(x) = (y_0, \dots, y_{l-1}) \) consists of \( l \) elements, there are at most \( (p+1)^{l-j} \leq (p+1)^l \) ways to choose the remaining \( l-j \) elements (where \( p+1 \) refers to the \( p \) elements in \( A \) and the element \( \star \)). Therefore,
\[
\# M_j \leq \frac{q!}{j!(q-j)!} \cdot \frac{l!}{(l-j)!} \cdot (p+1)^l.
\]

Similarly, there are \( \frac{q!}{(j-1)!(q-(j-1))!} \) ways to choose \( j-1 \) elements from \( B \), and there are \( \frac{l!}{(l-(j-1))!} \) ways to arrange each choice. Additionally, there are at most \( (p+1)^{l-j} \leq (p+1)^l \) ways to choose \( l-j \) elements from \( A \). Since each element of \( N_j \) consists of \( j-1 \) elements from \( B \), \( l-j \) elements from \( A \), and exactly one element from \( C \) (with this element being the last one appearing before the first \( \star \)), we have
\[
\#N_j \leq \frac{q!}{(j-1)!(q-(j-1))!} \cdot \frac{l!}{(l-(j-1))!} \cdot (p+1)^l \cdot r.
\]

Now, observe that
\[
\frac{1}{(j-1)!(q-(j-1))!(l-(j-1))!} = \frac{j}{j!(q-(j-1))!(l-(j-1))!} \leq \frac{j}{j!(q-j)!(l-j)!},
\]
which implies that \( \# N_j \leq j \cdot r \cdot \# M_j \).

Therefore, we have
\[
\#I_j = \# M_j + \# N_j \leq (jr+1) \cdot \frac{q!}{j!(q-j)!} \cdot \frac{l!}{(l-j)!} \cdot (p+1)^l\] \[ \leq ((q+1)r+1) \cdot \frac{q!}{j!(q-j)!} \cdot \frac{l!}{(l-j)!} \cdot (p+1)^l
\]
for each \( j \in \{0, \dots, q+1\} \).

Since \( \frac{q!}{j!(q-j)!} \leq q! \) and \( \frac{l!}{(l-j)!} \leq l^j \leq l^{q+1} \), we conclude that
\[
\# I_j \leq ((q+1)r+1) \cdot q! \cdot l^{q+1} \cdot (p+1)^l
\]
for each \( j \in \{0, \dots, q+1\} \). Consequently,
\[
\# \varphi_l(E) = \sum_{j=0}^{q+1} \# I_j \leq (q+2) \cdot ((q+1)r+1) \cdot q! \cdot l^{q+1} \cdot (p+1)^l\] \[ \leq (q+2)! \cdot ((q+1)r+1) \cdot l^{q+1} \cdot (p+1)^l.
\]
This completes the proof of Claim 7.
 $\square$

\vspace{0.5 cm}

Notice that by Claim 6, \(\varphi_l\) is injective on \(E\), so \(\# E = \#\varphi_l(E)\), and therefore \(\#E \leq (q+2)!((q+1)r+1)l^{q+1}(p+1)^l\).

Now fix \(n > 0\), let \(p\), \(q\), and \(r\) be as in Claim 7, and let \(m > 0\). Choose \(l_m \in \mathbb{N}\) such that \(n(l_m - 1) < m \leq nl_m\), and let \(E\) be as in Claim 7. Since \(F = \text{Dom}(\sigma^{n-1})\) is a closed set, we have
\[
ssep(m, 2\varepsilon, \sigma, X) = sep(m, 2\varepsilon, \sigma, F),
\]
and thus
$$
\renewcommand{\arraystretch}{2.3}
\begin{array}{rcl} 
\displaystyle \frac{ssep(m, 2\varepsilon, \sigma, X)}{m} & = & \displaystyle \frac{sep(m, 2\varepsilon, \sigma, F)}{m} = \frac{\log(\# E)}{m} \leq \frac{\log((q+2)!((q+1)r+1)l_m^{q+1}(p+1)^{l_m})}{m} \\
 & \leq & \displaystyle \frac{\log((q+2)!((q+1)r+1)l_m^{q+1}(p+1)^{l_m})}{n(l_m-1)} \\
 & = & \frac{1}{n(l_m-1)}\left(\log((q+2)!((q+1)r+1)) + (q+1)\log(l_m) + l_m\log(p+1)\right).
\end{array}
$$

Therefore, as \(m \rightarrow \infty\) and \(l_m \rightarrow \infty\), we have
\[
h_{2\varepsilon}(\sigma, X, d) = \limsup\limits_{m\rightarrow \infty} \frac{1}{m}ssep(m, 2\varepsilon, \sigma, X) \leq \frac{1}{n} \log(p+1).
\]
Since \(p+1 \leq 2p\), it follows that
\[
h_{2\varepsilon}(\sigma, X, d) \leq \frac{\log(2)}{n} + \frac{\log(p)}{n}
\]
for each \(n \in \mathbb{N}\).

Notice that \(p = sgen(n, \varepsilon, \sigma_\omega, \Omega)\), and from Proposition~\ref{span and gen} we have \(sgen(n, \varepsilon, \sigma_{|_\Omega}, \Omega) \leq ssep(n, \frac{\varepsilon}{2}, \sigma_{|_\Omega}, \Omega)\). Therefore,
\[
h_{2\varepsilon}(\sigma, X, d) \leq \limsup\limits_{n \rightarrow \infty} \left\{ \frac{\log(2)}{n} + \frac{\log(ssep(n, \frac{\varepsilon}{2}, \sigma_{|_\Omega}, \Omega))}{n} \right\} = h_{\frac{\varepsilon}{2}}(\sigma_{|_\Omega}, \Omega, d).
\]
Thus, we conclude that \(h_{2\varepsilon}(\sigma, X, d) \leq h_{\frac{\varepsilon}{2}}(\sigma_{|_\Omega}, \Omega, d)\) for each \(\varepsilon > 0\), and consequently, \(h_d(\sigma) \leq h_d(\sigma_{|_\Omega})\), which proves the theorem.

\end{proof}

\begin{example}

Consider Example~\ref{Cantor}, where \(X\) is the Cantor set on the real line, and \(\sigma\) is the local homeomorphism defined by multiplication by 3, with \(\Dom(\sigma) = [0,\frac{1}{3}] \cap X\). Notice that \(\Dom(\sigma)\) is clopen. Following a similar approach as in Example \ref{exemplo_fecho_de_W}, it can be shown that \(\Omega = \{0\}\). By applying Theorem \ref{entropias iguais}, we then conclude that \(h_d(\sigma_{|_\Omega}) = 0\).

\end{example}

 \begin{example} Let \( X = \left[-1, -\frac{1}{2}\right] \cup [0, 1] \), and define \(\sigma: X \rightarrow X\) by \(\sigma(x) = x^2\). In this case, \(\Omega = \{-1, 0, 1\}\) and \(\W = \left(-1, -\frac{1}{2}\right) \cup (0, 1)\). Since \(h_d(\sigma_{|_\Omega}) = 0\), it follows from Theorem \ref{entropias iguais} that \(h(\sigma) = 0\).

\end{example}

\begin{remark}
    
Let $(\R, \sigma)$ be the DR system described in Example~\ref{entropias diferentes}. Notice that in this case, \(h_d(\sigma) \neq h_d(\sigma_{|_\Omega})\); specifically, \(h_d(\sigma) = \log(2)\) and \(h_d(\sigma_{|_\Omega}) = 0\). This is an example where $\Dom(\sigma)$ is clopen, but $\R$ is not compact. Therefore, the assumption that \(X\) is compact cannot be omitted in Theorem~\ref{entropias iguais}.
    
\end{remark}

\begin{corollary}
    Let \((X, \sigma)\) be a compact DR system with \(\Dom(\sigma)\) clopen, and recall that \(\Omega\) and \(\W\) denote the nonwandering and wandering sets, respectively. Then \(h_d(\sigma_{|\overline{\W}}) \leq h_d(\sigma_{|_\Omega})\).
\end{corollary}

\begin{proof}
    This result follows from Theorems~\ref{entropias iguais} and \ref{entropias wandering e nonwandering}.
\end{proof}

\begin{example}
Let \( X = \prod \{0,1\} \) with the usual metric \( d(x,y) = \frac{1}{2^i} \), where \( i \) is the first position where the sequences \( x \) and \( y \) differ. Clearly, $X$ is compact. We define a DR system \( (X,f) \) as follows:

Decompose \( X \) as:
\[
X = [1] \sqcup [01] \sqcup [001] \sqcup [0001] \sqcup [00001] \sqcup \cdots \sqcup \{0^\infty\},
\]
where \( [a] \) denotes the cylinder set of all sequences that begin with \( a \). Let  \( \text{Dom}(f) = [0] \), and let \( \sigma \) denote the shift map. We define the map \( f \) as follows:
\[
f(0^\infty) = 0^\infty,
\]
\[
f|_{[0^{2n}1]}(x) = \sigma^2(x), \quad \text{for } n=1, 2, 3, \ldots,
\]
and
\[
f|_{[0^{2n+1}1]}(x) = 00x, \quad \text{for } n=0, 1, 2, \ldots.
\]

It is clear that \( f \) is an injective local homeomorphism on the clopen set \( \text{Dom}(f) \). Therefore, \( (X, f) \) is a DR system. Moreover, by Item 4 of Lemma~\ref{subset lemma}, we have that \( f^{-k}(f^k(U)) \subseteq U \) for all \( k \in \mathbb{Z} \) and for every \( U \subseteq X \).

Thus, if \( x \in [0^k1] \), then \( f^m([0^k1]) \cap f^{-k}(f^k([0^k1])) = \emptyset \) for all \( m > 0 \) and \( k \in \mathbb{Z} \). Hence, \( [0] \setminus \{0^\infty\} \subseteq \mathcal{W} \). Clearly, \( 0^\infty \in \Omega \), since \( 0^\infty \) is a fixed point.

Therefore, by Theorem~\ref{entropias iguais}, we conclude that \( h_d(f) = h_d(f|_\Omega) = 0 \).

\end{example}

\bibliographystyle{abbrv}
\bibliography{ref}

(Daniel Gon\c{c}alves, Danilo Royer and Felipe Augusto Tasca) 
 { Departamento de Matem\'atica, UNIVERSIDADE FEDERAL DE SANTA CATARINA, 88040-970, Florian\'opolis SC, Brazil.}
 
{\textit{Email Address:}}\texttt{
\href{mailto:daemig@gmail.com}{daemig@gmail.com}, \href{mailto:danilo.royer@ufsc.br}{danilo.royer@ufsc.br}, \href{mailto:tasca.felipe@gmail.com}{tasca.felipe@gmail.com}.}

\end{document}